\documentclass{article}
\usepackage{graphicx} 
\usepackage{hyperref}
\usepackage{amsmath,amssymb,color,amsthm,enumitem}
\usepackage{comment}
\usepackage{indentfirst}
\usepackage[margin=1.2in]{geometry}
\allowdisplaybreaks

\theoremstyle{plain}
\newtheorem{theorem}{Theorem}[section]
\newtheorem{lemma}[theorem]{Lemma}
\newtheorem{corollary}[theorem]{Corollary}
\newtheorem{proposition}[theorem]{Proposition}

\theoremstyle{definition}
\newtheorem{definition}[theorem]{Definition}
\newtheorem{claim}[theorem]{Claim}
\DeclareMathOperator{\osh}{osh}
\DeclareMathOperator{\sh}{sh}
\DeclareMathOperator{\st}{st}

\newcommand{\cF}{\mathcal{F}}
\newcommand{\cG}{\mathcal{G}}
\newcommand{\cS}{\mathcal{S}}
\newcommand{\cP}{\mathcal{P}}
\newcommand{\tm}{t_{min}}
\newcommand{\slm}{$\preceq$-minimal}
\newcommand\ceil[1]{\left\lceil#1\right\rceil}
\newcommand\floor[1]{\left\lfloor#1\right\rfloor}
\renewcommand{\emptyset}{\varnothing}

\title{More Shattering News}
\author{Attila Sali\\HUN-REN Alfr\'ed R\'enyi Institute of Mathematics and \\Department of Computer Science, BUTE \and Jun Yan\thanks{Supported by
ERC Advanced Grant 883810.}\\ Mathematical Institute, University of Oxford} 

\begin{document}

\maketitle

\begin{abstract}
An ordered variant of the well-known set theory concept of shattering was introduced by Anstee, R\'onyai, and Sali in~\cite{anstee2002shattering}. In this paper, we prove several new results related to order shattering. Given a family $\cF$ of subsets of $[n]$, we show that $\osh(\cF)$, the family of all sets order shattered by $\cF$, coincides with $T(\cF)$, the family obtained from $\cF$ by the down-shift operation. We then give a full characterization of all sets that can be order shattered by some $\ell$-Sperner family. Finally, we completely determine $\osh\left(\binom{[n]}{a}\cup\binom{[n]}{b}\right)$.
\end{abstract}

\section{Introduction}
Throughout the paper, we will use the standard notation $[n]$ to denote the set $\{1,2,\ldots,n\}$. For all integers $a,b$, we use $[a,b]$ to denote the set $\{a,a+1,a+2,\ldots,b\}$. In the degenerate case when $a>b$, $[a,b]$ is the empty set. For any set $A$, $\binom{A}{k}$ denotes the collection of all $k$-element subsets of $A$, and $\mathcal{P}(A)$ denotes the set of all subsets of $A$. 

We begin by recalling the set theory concept of shattering. 
\begin{definition}\label{def:shatter}
Let $S\subseteq [n]$ and let $\cF\subseteq\cP([n])$ be a family of subsets of $[n]$. 
\begin{itemize}
    \item $\cF$ \textit{shatters} $S$ if $\{F\cap S:F\in\cF\}=\cP(S)$. In other words, every subset of $S$ is the intersection of $S$ with some set in $\cF$. 
    \item The set of all sets shattered by $\cF$ is denoted by
\[\sh (\cF)=\{S\subseteq [n]:\cF\hbox{ shatters }S\}.\]
\end{itemize}
\end{definition}

It is clear from definition that $\sh(\cF)$ is a \textit{downset}, which means that if $S\in \sh(\cF)$ and $S'\subseteq S$, then $S'\in\sh(\cF)$. The well-known Sauer Inequality states that $|\sh(\cF)|\geq|\cF|$.

A more restrictive concept of strongly traced sets was introduced by Bollobás, Leader, and Radcliffe in \cite{bollobas1989reverse,bollobas1995defect}.
\begin{definition}[\cite{bollobas1989reverse,bollobas1995defect}]\label{def:stronglytraced}
Let $S\subseteq [n]$ and let $\cF\subseteq\cP([n])$ be a family of subsets of $[n]$. 
\begin{itemize}
    \item $S$ is \textit{strongly traced} by $\cF$ if there exists a set $B\subset[n]\setminus S$ such that $\{B\cup H: H\subseteq S\}\subseteq\cF$.
    \item The collection of all sets strongly traced by $\cF$ is denoted by $\st(\cF)$.
\end{itemize}
\end{definition}
It is clear that $\st(\cF)$ is also a downset. Furthermore, Bollobás, Leader, and Radcliffe proved the Reverse Sauer Inequality, which states that $|\st(\cF)|\le |\cF|$.

An intermediate concept of order shattering was introduced by Anstee, R\'onyai, and Sali in~\cite{anstee2002shattering}.
\begin{definition}[\cite{anstee2002shattering}]\label{def:osh-orig}
Let $S\subseteq [n]$ and let $\cF\subseteq\cP([n])$ be a family of subsets of $[n]$. 
\begin{itemize}
\item We inductively define what it means for $S$ to be \textit{order shattered} by $\cF$ as follows. 

If $S=\emptyset$, then $S$ is \textit{order shattered} by $\cF$ as long as $\cF\not=\emptyset$. 

If $|S|=k\geq1$, let the elements in $S$ be $s_1< s_2< \cdots < s_k$. We say that $S$ is \textit{order shattered} by $\cF$ if there exist disjoint families $\widetilde{\cF_0},\widetilde{\cF_1}\subseteq\cF$, each of size $2^{|S|-1}$, so that if we let $T=[s_k+1,n]$ (possibly $T=\emptyset$), then the following hold.
\begin{enumerate}[label=\alph*)]
    \item $T\cap C =T\cap D \hbox{ for all }C\in\widetilde{\cF_0}$ and $D\in \widetilde{\cF_1}$.
    \item $\{s_k\}\cap C =\emptyset$ and $\{s_k\}\cap D=\{s_k\}\hbox{ for all }C\in\widetilde{\cF_0}$ and $D\in \widetilde{\cF_1}$.
    \item Each of $\widetilde{\cF_0}$ and $\widetilde{\cF_1}$ individually order shatters $S\setminus\{s_k\}$.
\end{enumerate}

\item The set of all sets order shattered by $\cF$ is denoted by
\[\osh (\cF)=\{S\subseteq [n]:\cF\hbox{ order shatters }S\}.\]
\end{itemize}
\end{definition}
It follows from the definition that $\st(\cF)\subseteq\osh(\cF)\subseteq\sh(\cF)$, $\osh(\cF)$ is a downset,
$\osh(\osh(\cF))=\osh(\cF)$, and $\osh(\cF)=\osh(\cF^c)$, where $\cF^c$ is the family consisting of the complements of the sets in $\cF$.  Moreover, by~{\cite[Theorem 1.4]{anstee2002shattering}}, $|\osh(\cF)|=|\cF|$, which in particular proves the Reverse Sauer Inequality. 

The following is a non-inductive reformulation of Definition~\ref{def:osh-orig} that is often easier to work with.
\begin{proposition}\label{def:osh-direct}
A set $S\subseteq[n]$ with elements $s_1< s_2< \cdots < s_k$ is order shattered by $\cF\subseteq\cP([n])$ if and only if there exists a subfamily $\mathcal{G}=\{G_1,G_2,\ldots, G_{2^k}\}\subseteq\mathcal{F}$ of size $2^k$ such that the following hold.
\begin{enumerate}[label=\textup{\roman*)}]
    \item\label{osh:1} For all $i\in[k]$ and $j\in[2^k]$, $s_i\in G_j$ if and only if \[j\in\bigcup_{\ell=1}^{2^{k-i}}[\ell\cdot2^i-2^{i-1}+1,\ell\cdot 2^i].\]
    \item\label{osh:2} For all $i\in[k]$, $\ell\in[2^{k-i}]$, and all $(\ell-1)\cdot 2^i<a<b\le\ell\cdot 2^i$, \[G_a\cap[s_i+1,n]=G_b\cap[s_i+1,n].\]
\end{enumerate}
\end{proposition}
The order of the sets in the family $\cG$ in Proposition~\ref{def:osh-direct} is called \emph{the standard order} of $\cG$. For the sake of convenience, whenever a family $\cG$ of size $2^{|S|}$ order shatters a set $S$, $\cG$ is assumed to be in the standard order given by Proposition~\ref{def:osh-direct}. The condition~\ref{osh:1} can also be restated as follows. For every $j\in[2^k]$, if $\chi_j\in\{0,1\}^k$ is the characteristic vector of the set $S\cap G_j$, then $\chi_j$ is the binary representation of the integer $j-1$.

A connection  between order shattering and Gröbner bases was established in \cite{anstee2002shattering}, which has led to fruitful interactions between algebra, symbolic computation, and combinatorics. See for example,  \cite{friedl2003order,ell-wide,hegedHus2003grobner,felszeghy2006lex}.

The present paper is organized as follows.
First, in Section~\ref{sec:shift}, we give another characterization of $\osh(\cF)$ using the well-known down-shift operator. More precisely, in Theorem~\ref{thm:shift=osh}, we prove that the down-shifted family $T(\cF)$ obtained from $\cF$, with the down-shifts carried out in the natural order, is exactly equal to $\osh(\cF)$. 

Then, in Section~\ref{sec:ellSperner}, we generalize a simple characterization given in \cite{anstee2002shattering} for sets that can be order shattered by an antichain (Sperner family). We prove Theorem~\ref{thm:lspernercriterion}, which provides an analogous characterization for sets that can be order shattered by an $\ell$-Sperner family.

Finally, in Section~\ref{sec:2levels}, we determine $\osh\left(\binom{[n]}{a}\cup\binom{[n]}{b}\right)$, the sets order shattered by the union of two complete, but not necessarily consecutive, levels of $\cP([n])$. This builds on previous work in~\cite{anstee2002shattering} that determined the sets order shattered by one complete level of $\cP([n])$, and in~\cite{ell-wide} by Friedl, Heged\H us, and Rónyai that determined the sets order shattered by the union of any number of consecutive levels of $\cP([n])$.

\section{Shifting}\label{sec:shift}
Shifting proofs in extremal set theory were popularized by Peter Frankl~\cite{frankl2018extremal}. In this paper, we consider the down-shift operation that was used for example in~\cite{Alo83,Ans88,survey}. 
\begin{definition}
Let ${\mathcal F}\subseteq\cP([n])$. 
\begin{itemize}
    \item For every $j\in[n]$ and $B\in\cF$, let 
\[T_j(B)=\begin{cases}B&\hbox{if }j\notin B\hbox{ or }B\setminus\{j\}\in{\mathcal F}\\ B\setminus\{j\}&\hbox{if }j\in B\hbox{ and }B\setminus\{j\}\notin {\mathcal F}\end{cases},\]
and let $T_j({\mathcal F})=\{ T_j(B)\,:\,B\in{\mathcal F}\}$.
    \item The \emph{shifted family} $T({\mathcal F})$ is the family obtained by applying $T_1,T_2,\ldots,T_n$ in this order to $\cF$. 
\end{itemize}
\end{definition}
Note that for every $j\in[n]$, $T_j$ is injective, $|T_j({\mathcal F})|=|{\mathcal F}|$, and $T_j(T({\mathcal F}))=T({\mathcal F})$.
Thus, $|T({\mathcal F})|=|{\mathcal F}|$ and $T({\mathcal F})$ is a downset.

Recall from earlier that $|\osh(\mathcal{F})|=|\mathcal{F}|$ and $\osh(\mathcal{F})$ is also a downset for every ${\mathcal F}\subseteq\cP([n])$. In view of this, our following theorem stating that $\osh(\cF)=T(\cF)$ may not be surprising.
\begin{theorem}\label{thm:shift=osh}
    Let ${\mathcal F}\subseteq\cP([n])$, and let $T({\mathcal F})$ be the shifted family obtained from $\mathcal{F}$ by successively applying shifts $T_j$ for $j=1,2,\ldots,n$ in this order. Then,
    \[
    T({\mathcal F})=\osh(\mathcal{F}).
    \]
\end{theorem}
\begin{proof}
    Since $|T({\mathcal F})|=|\osh(\mathcal{F})|=|\cF|$, it is enough to show that $T({\mathcal F})\supseteq\osh(\mathcal{F})$.
Suppose then that $S=\{s_1,s_2,\ldots ,s_k\}\in\osh(\mathcal{F})$, with $s_1<s_2<\cdots<s_k$. By Proposition~\ref{def:osh-direct}, there exists a subfamily $\mathcal{G}=\{G_1,G_2,\ldots, G_{2^k}\}\subseteq\mathcal{F}$ of size $2^k$ in standard order satisfying~\ref{osh:1} and~\ref{osh:2}.

    For convenience, let $s_0=0$, $S_0=\emptyset$, $S_i=\{s_1,s_2,\ldots ,s_i\}$ for every $i\in[k]$, and $s_{k+1}=n+1$.
\begin{claim}\label{claim}
For every $h\in[n+1]$, let $0\leq i\leq k$ be the unique index satisfying $s_i<h\leq s_{i+1}$. Then, for every $j\in[2^k]$, \[G_j\cap(S_i\cup[h,n])\in T_{h-1}(T_{h-2}(\cdots (T_1(\mathcal{F}))\cdots)).\]
\end{claim}
\begin{proof}[Proof of Claim~\ref{claim}]
    We use induction on $h$. The base case when $h=1$ simply states $G_j\in\cF$ for every $j\in[2^k]$, and is thus trivial. Now suppose $h\in[n]$, and inductively assume that $G_j\cap(S_i\cup[h,n])\in T_{h-1}(T_{h-2}(\cdots (T_1(\mathcal{F}))\cdots))$ for every $j\in[2^k]$. 
    
    First, suppose that $h<s_{i+1}$, so $s_i<h<h+1\leq s_{i+1}$. If $h\not\in G_j$, then we have 
    \[G_j\cap(S_i\cup[h+1,n])=T_h(G_j\cap(S_i\cup[h,n]))\in T_h(T_{h-1}(\cdots (T_1(\mathcal{F}))\cdots)).\] If $h\in G_j$, then by definition and using $h\not\in S_i\cup[h+1,n]$, 
    \[
    (G_j\cap(S_i\cup[h,n]))\setminus\{h\}=G_j\cap(S_i\cup[h+1,n]).
    \]
    From the definition of $T_h$, either $(G_j\cap(S_i\cup[h,n]))\setminus\{h\}\in T_{h-1}(T_{h-2}(\cdots (T_1(\mathcal{F}))\cdots))$, or $T_h(G_j\cap(S_i\cup[h,n]))=(G_j\cap(S_i\cup[h,n]))\setminus\{h\}$. In either case, we have 
    \[
    G_j\cap(S_i\cup[h+1,n])=(G_j\cap(S_i\cup[h,n]))\setminus\{h\}\in T_h(T_{h-1}(\cdots (T_1(\mathcal{F}))\cdots)).
    \]

Now suppose $h=s_{i+1}$, so $s_{i+1}<h+1\leq s_{i+2}$. If $h\not\in G_j$, then $h\not\in G_j\cap(S_i\cup[h,n])$, so we have  
\[G_j\cap(S_{i+1}\cup[h+1,n])=T_h(G_j\cap(S_i\cup[h,n]))\in T_h(T_{h-1}(\cdots (T_1(\mathcal{F}))\cdots)).\]
 On the other hand, if $h=s_{i+1}\in G_j$, then by~\ref{osh:1} and~\ref{osh:2} of Proposition~\ref{def:osh-direct}, $G_j\cap(S_i\cup[h,n])$ and $G_{j-2^i}\cap(S_i\cup[h,n])$ only differ in the element $s_{i+1}$, with the former containing it and the latter not. This implies that $(G_j\cap(S_i\cup[h,n]))\setminus\{h\}=G_{j-2^i}\cap(S_i\cup[h,n])$. By induction hypothesis, both $G_{j}\cap(S_i\cup[h,n])$ and $G_{j-2^i}\cap(S_i\cup[h,n])$ are in $T_{h-1}(T_{h-2}(\cdots (T_1(\mathcal{F}))\cdots))$, so the definition of $T_h$ gives 
\[G_j\cap(S_{i+1}\cup[h+1,n])=G_j\cap(S_i\cup[h,n])=T_h(G_j\cap(S_i\cup[h,n]))\in T_h(T_{h-1}(\cdots (T_1(\mathcal{F}))\cdots)),\]
proving the claim.
\renewcommand{\qedsymbol}{$\boxdot$}
\end{proof}
\renewcommand{\qedsymbol}{$\square$}

Recall that $S=S_k\subseteq G_{2^k}$. Apply Claim~\ref{claim} with $h=n+1$, $i=k$, and $j=2^k$, we get
\[S=S_k=G_{2^k}\cap(S_k\cup\varnothing)\in T_n(T_{n-1}(\cdots (T_1(\mathcal{F}))\cdots))=T(\cF),\]
which finishes the proof of Theorem~\ref{thm:shift=osh}. 
\end{proof}

\section{$\ell$-Sperner families}\label{sec:ellSperner}
\begin{definition}
Let $\ell$ be a positive integer. $\cS\subseteq\cP([n])$ is a \emph{$\ell$-Sperner family} if there does not exist distinct $S_1,S_2,\ldots,S_{\ell+1}\in\cS$ such that $S_1\subset S_2\subset\cdots\subset S_{\ell+1}$.
\end{definition}
A 1-Sperner family is simply called a Sperner family, and is also frequently referred to in the literature as an \emph{antichain}. In~\cite{anstee2002shattering}, Anstee, R\'onyai, and Sali gave a simple characterization of all sets that can be order shattered by a Sperner family. 
\begin{theorem}[{\cite[Theorem 3.1]{anstee2002shattering}}]
Let $A=\{a_1,a_2,\ldots,a_k\}$ with $a_1<a_2<\cdots<a_k$. Then, there exists a Sperner family $\cS$ with $A\in\osh(\cS)$ if and only if
\[\sum_{i=1}^k\frac1{2^{a_i-i}}<1.\]
\end{theorem}

By adapting their methods, we generalize the result above and fully characterize the sets that can be order shattered by an $\ell$-Sperner family. 
\begin{theorem}\label{thm:lspernercriterion}
Let $A=\{a_1,a_2,\ldots,a_k\}$ with $a_1<a_2<\cdots<a_k$, and let $\ell$ be a positive integer. Then, there exists an $\ell$-Sperner family $\cS$ with $A\in\osh(\cS)$ if and only if
\[\sum_{i=1}^k\frac1{2^{a_i-i}}<\ell.\]
\end{theorem}

The proof of Theorem~\ref{thm:lspernercriterion} is quite technical. We begin with the following two lemmas that solve the case when $A$ is a set of consecutive integers, and illustrate the validity of the criterion. 
\begin{lemma}\label{lemma:j=1no}
Let $g$ be a positive integer. Then, there does not exist an $\ell$-Sperner family $\cS\subseteq\cP([n])$ such that \[[g+1,g+\ell\cdot2^g]\in\osh(\cS).\]
\end{lemma}
\begin{proof}
Let $A=[g+1,g+\ell\cdot2^g]$, and suppose for a contradiction that there is an $\ell$-Sperner family $\cS$ satisfying $A\in\osh(\cS)$. Then, by Proposition~\ref{def:osh-direct}, there exist $G_1,\ldots,G_{2^{\ell\cdot2^g}}\in\cS$, along with a decomposition $G_i=B_i\cup C_i\cup D_i$ for every $i\in[2^{\ell\cdot2^g}]$, such that the following hold.
\begin{itemize}
    \item $B_i\subseteq[g]$ for every $i\in[2^{\ell\cdot2^g}]$.
    \item $\{C_i:i\in[2^{\ell\cdot2^g}]\}=\cP([g+1,g+\ell\cdot2^g])$.
    \item There exists $D\subseteq[g+\ell\cdot2^g+1,n]$, such that $D_i=D$ for every $i\in[2^{\ell\cdot2^g}]$.
\end{itemize}
In particular, we can find distinct $i_1,\ldots,i_{\ell\cdot2^g+1}\in[2^{\ell\cdot2^g}]$, such that $C_{i_1}\subset C_{i_2}\subset\cdots\subset C_{\ell\cdot2^g+1}$. Then, by pigeonhole, there exists distinct $j_1,\ldots,j_{\ell+1}\in\{i_1,\ldots,i_{\ell\cdot2^g+1}\}$, such that $B_{j_1}=B_{j_2}=\cdots=B_{j_{\ell+1}}$. It follows that $G_{j_1}\subset G_{j_2}\subset\cdots\subset G_{j_{\ell+1}}$, which contradicts that $\cS$ is an $\ell$-Sperner family. 
\end{proof}

\begin{lemma}\label{lemma:j=1yes}
For every positive integer $g$, there exists an $\ell$-Sperner family $\cS\subseteq\cP([g+\ell\cdot2^g-1])$ with \[[g+1,g+\ell\cdot2^g-1]\in\osh(\cS).\]
\end{lemma}
\begin{proof}
Let $A=[g+1,g+\ell\cdot2^g-1]$. Order the $2^g$ subsets of $[g]$ as $G_1,G_2,\ldots,G_{2^g}$, such that $G_i\setminus G_j\not=\varnothing$ for any $1\leq i<j\leq2^g$. This can be achieved by any ordering satisfying $|G_1|\geq|G_2|\geq\cdots\geq|G_{2^g}|$. 
 
Consider the family $\cS$ of subsets of $[n]$ given by
\[\cS=\bigcup_{i=1}^{2^g}\left\{G_i\cup A':A'\in\bigcup_{j=(i-1)\ell}^{i\ell-1}\binom{A}j\right\}.\]
We claim that $\cS$ is an $\ell$-Sperner family and $A\in\osh(\cS)$. The latter is clear from definition after noting that 
\[\bigcup_{i=1}^{2^g}\bigcup_{j=(i-1)\ell}^{i\ell-1}\binom{A}j=\bigcup_{j=0}^{\ell\cdot2^g-1}\binom{A}j=\cP(A).\]

Suppose for a contradiction that there exist distinct $S_1,S_2,\ldots,S_{\ell+1}\in\cS$, such that $S_1\subset S_2\subset\cdots\subset S_{\ell+1}$. Using the definition of $\cS$, for each $k\in[\ell+1]$, write $S_k$ as the disjoint union of $G_{i_k}$ and $A_k$, where $i_k\in[2^g]$ and $A_k\subseteq A$ satisfies $(i_k-1)\ell\leq|A_k|\leq i_k\ell-1$. It follows that $A_1\subseteq A_2\subseteq\cdots\subseteq A_{\ell+1}$, and $G_{i_1}\subseteq G_{i_2}\subseteq\cdots\subseteq G_{i_{\ell+1}}$, so $i_1\geq i_2\geq\cdots\geq i_{\ell+1}$. If $i_1=i_2=\cdots=i_{\ell+1}$, then by pigeonhole, there exists distinct $k_1,k_2\in[\ell+1]$, such that $|A_{k_1}|=|A_{k_2}|$, so $A_{k_1}=A_{k_2}$. But then, $S_{k_1}=S_{k_2}$, a contradiction. Therefore, there exists $k^*\in[\ell]$, such that $i_{k^*}>i_{k^*+1}$, but then $|A_{k^*}|\geq(i_{k^*}-1)\ell>i_{k^*+1}\ell-1\geq|A_{k^*+1}|$, contradicting $A_{k^*}\subseteq A_{k^*+1}$. Therefore, $\cS$ is an $\ell$-Sperner family.
\end{proof}

We now prove a series of intermediate results, each of which says that given a set $A$ that can be order shattered by an $\ell$-Sperner family, there is a certain operation that can be performed on $A$ to obtain another set $A'$ that can also be order shattered by an $\ell$-Sperner family. 

The first such result is similar to Lemma~\ref{lemma:j=1yes}, and says that if we leave a gap of size $g$, then we can add a new block of $2^g-1$ consecutive integers to the end of $A$. 
\begin{lemma}\label{lemma:extend}
Let $1\leq a_1<a_2<\cdots<a_t$ be integers, and let $A=\{a_1,\ldots,a_t\}$. Suppose there exists an $\ell$-Sperner family $\cS$ with $A\in\osh(\cS)$. For every positive integer $g$, let $A'=A\cup[a_t+g+1,a_t+g+2^g-1]$. Then, there exists an $\ell$-Sperner family $\cS'$ with $A'\in\osh(\cS')$.
\end{lemma}
\begin{proof}
Let $\cS\subseteq\cP([a_t])$ be an $\ell$-Sperner family satisfying $A\in\osh(\cS)$. Order the $2^g$ subsets of $[a_t+1,a_t+g]$ as $G_0,G_1,\ldots,G_{2^g-1}$, such that $G_i\setminus G_j\not=\varnothing$ for any $0\leq i<j\leq 2^g-1$. Let $B=[a_t+g+1,a_t+g+2^g-1]$, and consider the family $\cS'$ of sets given by
\[\cS'=\bigcup_{i=0}^{2^g-1}\left\{S\cup G_i\cup B':S\in\cS, B'\in\binom{B}i\right\}.\]

From definition, it is clear that $A'=A\cup B\in\osh(\cS')$, so it remains to show that $\cS'$ is an $\ell$-Sperner family. Indeed, suppose for a contradiction that there exist distinct $S_1',S_2',\ldots,S_{\ell+1}'\in\cS'$, such that $S_1'\subset S_2'\subset\cdots\subset S_{\ell+1}'$. Using the definition of $\cS'$, for each $k\in[\ell]$, write $S_k'$ as the disjoint union of $S_k$, $G_{i_k}$, and $B_k'$, where $S_k\in\cS$, $0\leq i_k\leq 2^g-1$, and $B_k'\in\binom{B}{i_k}$. It follows that $S_1\subseteq S_2\subseteq\cdots\subseteq S_{\ell+1}$, so there exists $k^*\in[\ell+1]$ such that $S_{k^*}=S_{k^*+1}$, as $\cS$ is an $\ell$-Sperner family. Moreover, from $G_{i_{k^*}}\subseteq G_{i_{k^*}+1}$, it follows that $i_{k^*}\geq i_{k^*+1}$ and $|B_{k^*}'|=i_{k^*}\geq i_{k^*+1}=|B'_{k^*+1}|$. However, $B_{k^*}'\subseteq B_{k^*+1}'$, so $i_{k^*}=i_{k^*+1}$ and $B_{k^*}'=B_{k^*+1}'$. Therefore, $S_{k^*}'=S_{k^*+1}'$, a contradiction.
\end{proof}

The next two results together show that if the largest $r$ numbers in $A$ are consecutive, then by increasing each of them by $g$, we can further add on the next $(r+1)2^g-r-1$ numbers.
\begin{lemma}\label{lemma:shiftandextend}
Let $1\leq a_1<a_2<\cdots<a_t$ and $r,g\geq1$ be integers. Let $A=\{a_1,\ldots,a_t,a_t+g+1,\ldots,a_t+g+r\}$ and $A'=\{a_1,\ldots,a_{t-1},a_t+g,a_t+g+1,\ldots,a_t+g+2^g+\ceil{(r+1)/2^g}2^g-2\}$. If there exists an $\ell$-Sperner family $\cS$ with $A\in\osh(\cS)$, then there exists an $\ell$-Sperner family $\cS'$ with $A'\in\osh(\cS')$.
\end{lemma}
\begin{proof}
Let $A_1=\{a_1,\ldots,a_{t-1}\}$ and $A_2=[a_t+g+1,a_t+g+r]$. Let $\cS\subseteq\cP([a_t+g+r])$ be an $\ell$-Sperner family satisfying $|\cS|=2^{t+r}$ and $A\in\osh(\cS)$. From definition, for every $C\subseteq A_2$, there exist two families $\cS_{C,1},\cS_{C,2}\subseteq\cP([a_t-1])$ with $|\cS_{C,1}|=|\cS_{C,2}|=2^{t-1}$ and $A_1\in\osh(\cS_{C,1})\cap\osh(\cS_{C,2})$, as well as a set $B_C\subset[a_t+1,a_t+g]$, such that $\cS$ is the disjoint union of \[\cS(C)=\{X\cup B_C\cup C:X\in\cS_{C,1}\}\cup\{X\cup\{a_t\}\cup B_C\cup C:X\in\cS_{C,2}\},\]
taken over all $C\subseteq A_2$. Note that both $\cS_{C,1}$ and $\cS_{C,2}$ are $\ell$-Sperner families, as any chain of length $\ell+1$ within one of them would also induce a chain of length $\ell+1$ in $\cS(C)\subseteq\cS$.

Let $p=\ceil{(r+1)/2^g}$. By pigeonhole, we can find distinct $C_1\subset C_2\subset\cdots\subset C_p\subseteq A_2$, such that $B_{C_1}=\cdots=B_{C_p}$. Let $\cS_i=\cS_{C_i,1}$ for every $i\in[p]$, and let $\cS_{p+1}=\cS_{C_p,2}$. Let $G_1,G_2,\ldots,G_{2^g}$ be an ordering of the $2^g$ subsets of $[a_t,a_t+g-1]$, such that for any $1\leq i<j\leq 2^g$, $G_i\setminus G_j\not=\varnothing$. Let $A_3=[a_t+g,a_t+g+(p+1)2^g-2]$. Let $\cS'$ be the family given by
\[\cS'=\bigcup_{j=1}^{2^g}\bigcup_{i=1}^{p+1}\left\{X\cup G_j\cup Y:X\in\cS_i, Y\in\binom{A_3}{(j-1)(p+1)+i-1}\right\}.\]
Note that \[\bigcup_{j=1}^{2^g}\bigcup_{i=1}^{p+1}\binom{A_3}{(j-1)(p+1)+i-1}=\cP(A_3),\]
from which it follows that $A'=A_1\cup A_3\in\osh(\cS')$. 

Furthermore, we claim that $\cS'$ is an $\ell$-Sperner family. Indeed, suppose for a contradiction that there exist distinct $S_1'\subset S_2'\subset\cdots\subset S_{\ell+1}'$ in $\cS'$, where for each $k\in[\ell+1]$, $S_k'=X_{k}\cup G_{j_k}\cup Y_k$ for some $i_k\in[p+1]$, $X_{k}\in\cS_{i_k}$, $j_k\in[2^g]$, and $Y_k\subseteq A_3$ of size $(j_k-1)(p+1)+i_k-1$. Then, $G_{j_1}\subseteq\cdots\subseteq G_{j_{\ell+1}}$ implies that $j_1\geq\cdots\geq j_{\ell+1}$. Combined with $Y_1\subseteq\cdots\subseteq Y_{\ell+1}$, we get $j_1=\cdots=j_{\ell+1}$ and $i_1\leq\cdots\leq i_{\ell+1}$. 

We also know that $X_1\subseteq\cdots\subseteq X_{\ell+1}$. For every $k\in[\ell+1]$, if $i_k\in[p]$, let $S_k=X_k\cup B_{C_{i_k}}\cup C_{i_k}$, while if $i_k=p+1$, let $S_k=X_k\cup\{a_t\}\cup B_{C_p}\cup C_p$. We claim that $S_1\subset\cdots\subset S_{\ell+1}$, which leads to a contradiction as they form a chain of length $\ell+1$ in $\cS$. For every $k\in[\ell]$, if $i_k=i_{k+1}$, then $Y_k=Y_{k+1}$, so $X_k\subset X_{k+1}$ as $S_k'\subset S_{k+1}'$, which then implies $S_k\subset S_{k+1}$. If $i_k<i_{k+1}\leq p$, then $S_k\subset S_{k+1}$ as $C_{i_k}\subset C_{i_{k+1}}$. Finally, if $i_k<i_{k+1}=p+1$, then $S_k\subset S_{k+1}$ as $a_t\in S_{k+1}\setminus S_k$, finishing the proof. 
\end{proof}

\begin{corollary}\label{cor:extend}
Let $1\leq a_1<a_2<\cdots<a_t$ and $r\geq1$ be integers, and let $A=\{a_1,\ldots,a_t,a_t+1,\ldots,a_t+r\}$. Suppose there exists an $\ell$-Sperner family $\cS$ with $A\in\osh(\cS)$. For every positive integer $g$, let $A'=\{a_1,\ldots,a_t,a_t+g+1,a_t+g+2,\ldots,a_t+g+(r+1)2^g-1\}$. Then, there exists an $\ell$-Sperner family $\cS'$ with $A'\in\osh(\cS')$.
\end{corollary}
\begin{proof}
By induction on $j$, we show that for every $0\leq j\leq r$, there exists an $\ell$-Sperner family $\cS'_j$ that order shatters $A_j=\{a_1,\ldots,a_t,a_t+1,\ldots,a_t+r-j,a_t+r+g-j+1,\ldots,a_t+r+g-j+(j+1)2^g-1\}$. The base case when $j=0$ follows from Lemma~\ref{lemma:extend}, while the case when $j=r$ is what we need. Suppose this is true for some $0\leq j\leq r-1$, then the $\ell$-Sperner family $\cS'_{j+1}$ that order shatters $A_{j+1}$ is given by Lemma~\ref{lemma:shiftandextend}. 
\end{proof}

The next two results can be viewed as the reverse of the previous two. They together show that if there is a gap of at least $g$ between the last two blocks of consecutive integers in $A$, then we can reduce the gap size by $g$ at the cost of reducing the size of the last block roughly by a factor of $2^g$.
\begin{lemma}\label{lemma:shiftback}
Let $1\leq a_1<a_2<\cdots<a_t$ and $r\geq2$ be integers. Let $A=\{a_1,\ldots,a_{t-1},a_t+1,a_t+2,\ldots,a_t+r\}$ and $A'=\{a_1,\ldots,a_{t-1},a_t,a_t+2,a_t+3,\ldots,a_t+2\ceil{(r+1)/2}-2\}$. Suppose there exists an $\ell$-Sperner family $\cS$ with $A\in\osh(\cS)$. Then, there exists an $\ell$-Sperner family $\cS'$ with $A'\in\osh(\cS')$.
\end{lemma}
\begin{proof}
Let $A_1=\{a_1,\ldots,a_{t-1}\}$ and $A_2=[a_t+1,a_t+r]$. Let $\cS\subseteq\cP([a_t+r])$ be an $\ell$-Sperner family satisfying $|\cS|=2^{t+r-1}$ and $A\in\osh(\cS)$. From definition, for every $C\subseteq A_2$, there exists a family $\cS_C\subseteq\cP([a_t-1])$ with $|\cS_C|=2^{t-1}$ and $A_1\in\osh(\cS_C)$, as well as a set $B_C\subseteq\{a_t\}$, such that $\cS$ is the disjoint union of \[\cS(C)=\{X\cup B_C\cup C:X\in\cS_C\},\]
taken over all $C\subseteq A_2$. Note that each $\cS_C$ is an $\ell$-Sperner family, as any chain of length $\ell+1$ in $\cS_C$ would also induce a chain of length $\ell+1$ in $\cS(C)\subseteq\cS$.

Let $p=\ceil{(r+1)/2}$. By pigeonhole, we can find $C_1\subset C_2\subset\cdots\subset C_p\subseteq A_2$, such that $B_{C_1}=B_{C_2}=\cdots=B_{C_p}$. Let $A_3=[a_t+2,a_t+2p-2]$. Consider the family $\cS'$ given by taking the union of the following four families. 
\[\cS_1=\bigcup_{i=1}^{p-1}\left\{X\cup\{a_t+1\}\cup Y:X\in\cS_{C_i}, Y\in\binom{A_3}{i-1}\right\},\]
\[\cS_2=\bigcup_{i=1}^{p-1}\left\{X\cup\{a_t,a_t+1\}\cup Y:X\in\cS_{C_{i+1}}, Y\in\binom{A_3}{i-1}\right\},\]
\[\cS_3=\bigcup_{i=1}^{p-1}\left\{X\cup Y:X\in\cS_{C_i}, Y\in\binom{A_3}{p+i-2}\right\},\]
\[\cS_4=\bigcup_{i=1}^{p-1}\left\{X\cup\{a_t\}\cup Y:X\in\cS_{C_{i+1}}, Y\in\binom{A_3}{p+i-2}\right\}.\]

One can verify that $A'=A_1\cup\{a_t\}\cup A_3\in\osh(\cS')$. We now prove that $\cS'$ is an $\ell$-Sperner family. Suppose for a contradiction that there exist distinct $Z_1\subset Z_2\subset\cdots\subset Z_{\ell+1}$ in $\cS'$. Observe that every set in $\cS_1\cup\cS_2$ contains $a_t+1$, while no set in $\cS_3\cup\cS_4$ contains it. Also, every set in $\cS_3\cup\cS_4$ has strictly larger intersection with $A_3$ than every set in $\cS_1\cup\cS_2$. Therefore, every set in $\cS_1\cup\cS_2$ is incomparable with every set in $\cS_3\cup\cS_4$. Assume from now on that $Z_j\in\cS_1\cup\cS_2$ for every $j\in[\ell+1]$. The other case when $Z_j\in\cS_3\cup\cS_4$ for every $j\in[\ell+1]$ is analogous. 

Since every set in $\cS_2$ contains $a_t$, while no set in $\cS_1$ contains it, there exists some $0\leq m\leq\ell+1$, such that $Z_1,\ldots ,Z_m\in\cS_1$ and $Z_{m+1},\ldots ,Z_{\ell+1}\in\cS_2$. 
For every $j\in[m]$, say $Z_j=X_j\cup\{a_t+1\}\cup Y_j$ for some $i_j\in[p-1]$, $X_j\in S_{C_{i_j}}$ and $Y_j\in \binom{A_3}{i_j-1}$. For every $m+1\leq j\leq\ell+1$, say $Z_j=X_j\cup\{a_t,a_t+1\}\cup Y_j$ for some $i_j\in[p-1]$, $X_j\in S_{C_{i_j+1}}$ and $Y_j\in \binom{A_3}{i_j-1}$. Since $Z_1\subset Z_2\subset\cdots\subset Z_{\ell+1}$, we have $X_1\subseteq X_2\subseteq\cdots\subseteq X_{\ell+1}$, and $i_1\le i_2\le \cdots \le i_{\ell+1}$. Now, consider the sets $S_j=X_j\cup B_{C_{i_j}}\cup C_{i_j}$ for every $j\in[m]$, and $S_j=X_j\cup B_{C_{i_j+1}}\cup C_{i_j+1}$ for every $m+1\leq j\leq\ell+1$. Observe that $X_j=X_{j+1}$ and $i_j=i_{j+1}$ cannot happen at the same time for any $j\in[m-1]$ or any $m+1\le j \le \ell$, as it would imply $|Z_j|=|Z_{j+1}|$, a contradiction. Thus, $S_1\subset S_2\subset\cdots \subset S_m$ and $S_{m+1}\subset\cdots\subset S_{\ell+1}$. Furthermore, we always have $S_m\subset S_{m+1}$ as $C_{i_m}\subset C_{i_{m+1}+1}$. Therefore, $S_1\subset S_2\subset\cdots \subset S_{\ell+1}$ is a chain of length $\ell+1$ in $\cS$, contradicting that $\cS$ is an $\ell$-Sperner family. 
\end{proof}

\begin{corollary}\label{cor:shiftback}
Let $1\leq a_1<a_2<\cdots<a_t$ be integers, and let $r,g$ be positive integers satisfying $r\geq2^g$. Let $A=\{a_1,\ldots,a_{t-1},a_t+g,a_t+g+1\ldots,a_t+g+r-1\}$, and let $A'=\{a_1,\ldots,a_{t-1},a_t,a_t+1,\ldots,a_t+\ceil{(r+1)/2^g}-2\}$. Suppose there exists an $\ell$-Sperner family $\cS$ with $A\in\osh(\cS)$. Then, there exists an $\ell$-Sperner family $\cS'$ with $A'\in\osh(\cS')$.
\end{corollary}
\begin{proof}
For every $0\leq j\leq g$, we use induction on $j$ to show that there exists an $\ell$-Sperner family $\cS_j$ that order shatters $A_j=\{a_1,\ldots,a_{t-1},a_t+g-j,a_t+g-j+1,\ldots,a_t+g-j+\ceil{(r+1)/2^j}-2\}$. The case when $j=0$ follows from assumption, and the case when $j=g$ is what we need. 

Suppose $0\leq j\leq g-1$ and $A_j\in\osh(\cS_j)$ for some $\ell$-Sperner family $\cS_j$. By iteratively applying Lemma~\ref{lemma:shiftback}, we can always reduce the first element in the ending block of consecutive integers by 1, and additionally remove the last element if this block has odd length at least 3. Then, after $\ceil{\ceil{(r+1)/2^j}/2}-1=\ceil{(r+1)/2^{j+1}}-1$ steps, we can find an $\ell$-Sperner family $\cS_{j+1}$ such that $A_{j+1}\cup\{a_t+g-j-1+\ceil{(r+1)/2^{j+1}}\}\in\osh(\cS_{j+1})$, and so $A_{j+1}\in\osh(\cS_{j+1})$ as well.
\end{proof}

Finally, we can put everything together and prove Theorem~\ref{thm:lspernercriterion}. The idea is to view $A$ as blocks of consecutive integers, and use induction on the number of blocks. The preceding series of lemmas will be used in the induction step to create, delete, or modify the last block in $A$ as required.
\begin{proof}[Proof of Theorem~\ref{thm:lspernercriterion}]
Let $A$ be a set of positive integers. By breaking $A$ down into maximal blocks of consecutive integers, we can uniquely represent the elements in $A$ in increasing order as $g_1+1,\ldots,g_1+b_1,g_1+b_1+g_2+1,\ldots,g_1+b_1+g_2+b_2,\ldots,g_1+b_1+\cdots+g_j+1,\ldots,g_1+b_1+\cdots+g_j+b_j$, for some $j\geq 1$, $g_1,\ldots,g_j\geq0$, and $b_1,\ldots,b_j\geq1$. We say that $A$ is encoded by the sequence $g_1,b_1,\ldots,g_j,b_j$. 

To prove Theorem~\ref{thm:lspernercriterion}, we use induction on $j$. The base case when $j=1$ follows from  Lemma~\ref{lemma:j=1no} and Lemma~\ref{lemma:j=1yes}. Now let $j>1$.  

\textbf{Case 1.} $b_j\leq2^{g_j}-1$. Suppose there is an $\ell$-Sperner family $\cS$ such that $A\in\osh(\cS)$. Let $A'$ be the set encoded by $g_1,b_1,\ldots,g_{j-1},b_{j-1}$, then $A'\subseteq A$ is in $\osh(\cS)$ as well. Note that the last element in $A'$ is $a'_{b_1+\cdots+b_{j-1}}=g_1+b_1+\cdots+g_{j-1}+b_{j-1}$, so by induction hypothesis,
\[\sum_{i=1}^{b_1+\cdots+b_{j-1}}\frac1{2^{a_i'-i}}\leq\ell-\frac1{2^{g_1+\cdots+g_{j-1}}}.\]
Therefore, \[\sum_{i=1}^{b_1+\cdots+b_j}\frac1{2^{a_i-i}}\leq\ell-\frac1{2^{g_1+\cdots+g_{j-1}}}+\frac{b_j}{2^{g_1+\cdots+g_j}}\leq\ell-\frac1{2^{g_1+\cdots+g_j}}<\ell.\]

On the other hand, if $\sum_{i=1}^{b_1+\cdots+b_j}\frac1{2^{a_i-i}}<\ell$, then $\sum_{i=1}^{b_1+\cdots+b_{j-1}}\frac1{2^{a_i-i}}<\ell$ as well. Thus, by induction hypothesis, there is an $\ell$-Sperner family $\cS$ satisfying $A'\in\osh(\cS)$. Let $A''$ be the set encoded by $g_1,b_1,\ldots,g_j,2^{g_j}-1$. Then, by Lemma~\ref{lemma:extend}, there exists an $\ell$-Sperner family $\cS'$ such that $A''\in\osh(\cS')$. Since $b_j\leq 2^{g_j}-1$, $A\subseteq A''$ is in $\osh(\cS')$ as well.

\textbf{Case 2.} $b_j\geq2^{g_j}$. Suppose there is an $\ell$-Sperner family $\cS$ such that $A\in\osh(\cS)$. Let $b=\ceil{(b_j+1)/2^{g_j}}$, and let $A'$ be the set encoded by $g_1,b_1,\ldots,g_{j-1},b_{j-1}+b-1$. Then, by Corollary~\ref{cor:shiftback}, there exists an $\ell$-Sperner family $\cS'$ such that $A'\in\osh(\cS')$. Thus, by induction hypothesis,
\begin{align*}
\sum_{i=1}^{b_1+\cdots+b_j}\frac1{2^{a_i-i}}&=\frac{b_j}{2^{g_1+\cdots+g_j}}+\sum_{i=1}^{b_1+\cdots+b_{j-1}}\frac1{2^{a_i-i}}\\
&\leq\frac{b2^{g_j}-1}{2^{g_1+\cdots+g_j}}+\sum_{i=1}^{b_1+\cdots+b_{j-1}}\frac1{2^{a_i-i}}\\
&=\frac{2^{g_j}-1}{2^{g_1+\cdots+g_j}}+\sum_{i=1}^{b_1+\cdots+b_{j-1}+b-1}\frac1{2^{a_i'-i}}\\
&\leq\frac{2^{g_j}-1}{2^{g_1+\cdots+g_j}}+\ell-\frac1{2^{g_1+\cdots+g_{j-1}}}=\ell-\frac1{2^{g_1+\cdots+g_j}},
\end{align*}
where in the last inequality we used that the last element in $A'$ is $a'_{b_1+\cdots+b_{j-1}+b-1}=g_1+b_1+\cdots+g_{j-1}+b_{j-1}+b-1$.

On the other hand, suppose that $\sum_{i=1}^{b_1+\cdots+b_j}\frac1{2^{a_i-i}}<\ell$, then \[\sum_{i=1}^{b_1+\cdots+b_j}\frac1{2^{a_i-i}}\leq\ell-\frac1{2^{g_1+\cdots+g_j}}\] 
as the last element in $A$ is $a_{b_1+\cdots+b_j}=g_1+b_1+\cdots+g_j+b_j$. It follows that for $A'$,
\[\sum_{i=1}^{b_1+\cdots+b_{j-1}+b-1}\frac1{2^{a_i'-i}}\leq\ell-\frac{b_j+1}{2^{g_1+\cdots+g_j}}+\frac{b-1}{2^{g_1+\cdots+g_{j-1}}}<\ell.\]
Hence, by induction hypothesis, there exists an $\ell$-Sperner family $\cS$ with $A'\in\osh(\cS)$. Let $A''$ be the set encoded by $g_1,b_1,\ldots,g_{j-1},b_{j-1},g_j,b2^{g_j}-1$, and note that $A\subseteq A''$. Then, by Corollary~\ref{cor:extend}, there exists an $\ell$-Sperner family $\cS'$ with $A''\in\osh(\cS')$, and hence with $A\in\osh(\cS')$ as well.
\end{proof}

\section{Two complete levels}\label{sec:2levels}
In this section, we study $\osh(\cF)$ when $\cF$ is a union of complete levels in $\cP([n])$. To state our results in this section, we first need to generalize the well-known concept of ballot sets. 
\begin{definition}
     Let $t\geq 0$ be an integer, and let $S\subseteq [n]$ be a subset with elements $s_1<s_2<\cdots<s_k$.
     \begin{itemize}
         \item $S$ is called a \emph{$t$-ballot set} if for every $i\in[k]$, \[s_i\geq2i-t.\] Equivalently, $S$ is a $t$-ballot set if $|S\cap [m]|\le\floor{\frac{m+t}{2}}$ for every $m\in[n]$.
        \item The smallest non-negative integer $t$ such that $S$ is a $t$-ballot set is denoted by $\tm(S)$.
     \end{itemize}
\end{definition}
Note that a $0$-ballot set is just a ballot set in the classical sense~\cite{anstee2002shattering}. Furthermore, if $S$ is a $t$-ballot set and $t<t'$, then $S$ is a $t'$-ballot set as well. 

One of the main results of \cite{anstee2002shattering} is the determination of $\osh(\cF)$ when $\cF=\binom{[n]}k$ is a complete level of $\cP([n])$. This was extended in \cite{ell-wide} to the case when $\cF$ is the union of $\ell$ consecutive complete levels in $\cP([n])$ for any positive integer $\ell$. 
\begin{theorem}[{\cite[Theorem 4.1]{ell-wide}}]\label{thm:ellwide}
Let $0\leq\ell-1\leq k\leq n$.
\begin{itemize}
    \item If $k<(n+\ell)/2$, then
    \[\osh\left(\bigcup_{i=k-\ell+1}^k\binom{[n]}i\right)=\{A\subseteq[n]:|A|\leq k, \text{and }A\text{ is }(\ell-1)\text{-ballot}\}.\]
    \item If $k\geq(n+\ell)/2$, then
    \[\osh\left(\bigcup_{i=k-\ell+1}^k\binom{[n]}i\right)=\{A\subseteq[n]:|A|\leq n-k+\ell-1, \text{and }A\text{ is }(\ell-1)\text{-ballot}\}.\]
\end{itemize}
\end{theorem}

In this section, we consider the case when $\cF=\binom{[n]}{a}\cup\binom{[n]}{a+d}$ is the union of two complete, but non-consecutive levels.

Recall that $\osh(\cF)$ is a downset for any $\cF\subseteq\cP([n])$. Equivalently, if $S\in\osh(\cF)$ and $S'\subseteq S$, then $S'\in\osh(\cF)$ as well. We first show that in the more general case when $\cF=\binom{[n]}{a_1}\cup\binom{[n]}{a_2}\cup\cdots\cup\binom{[n]}{a_s}$ is a union of not necessarily consecutive complete levels, $\osh(\cF)$ is an upset with respect to another well-known partial order $\preceq$ on $\cP([n])$ defined as follows. Note that this is an extended version of the poset $L(k,n)$ introduced by Stanley \cite{doi:10.1137/0601021}.

\begin{definition}\label{def:prec}
Let $S$ and $S'$ be subsets of $[n]$. 
\begin{itemize}
    \item If $|S|\not=|S'|$, then $S$ and $S'$ are incomparable under $\preceq$.
    \item If $S$ and $S'$ both have size $k$, with elements $s_1<s_2<\cdots<s_k$ and $s_1'<s_2'<\cdots<s_k'$, respectively, then $S\preceq S'$ if and only if $s_i\leq s_i'$ for every $i\in[k]$.
    \item For convenience, if $S\preceq S'$ and $S\not=S'$, then we denote this by $S\prec S'$.
\end{itemize}
\end{definition}

The following lemma and its corollary show that if $\cF$ is a union of complete levels, $S\in\osh(\cF)$, and $S\preceq S'$, then $S'\in\osh(\cF)$ as well. Therefore, $\osh(\cF)$ is an upset with respect to $\preceq$ in this case. 
\begin{lemma}\label{lem:SL}
Let $0\le a_1<a_2<\cdots <a_s\le n$. Suppose $S\in\osh\left(\binom{[n]}{a_1}\cup\binom{[n]}{a_2}\cup\cdots\cup\binom{[n]}{a_s}\right)$, and the elements in $S$ are $s_1<s_2<\cdots<s_k$. If $i<k$ and $s_i+1<s_{i+1}$, or $i=k$ and $s_i<n$, then $S'=(S\setminus\{s_i\})\cup\{s_i+1\}$ also satisfies $S'\in\osh\left(\binom{[n]}{a_1}\cup\binom{[n]}{a_2}\cup\cdots\cup\binom{[n]}{a_s}\right)$.
\end{lemma}
\begin{proof}
Let $\mathcal{G}=\{G_1,G_2,\ldots,G_{2^k}\}\subseteq\binom{[n]}{a_1}\cup\binom{[n]}{a_2}\cup\cdots\cup\binom{[n]}{a_s}$ be a family in standard order satisfying $S\in\osh(\cG)$. Define $\cG'\subseteq\binom{[n]}{a_1}\cup\binom{[n]}{a_2}\cup\cdots\cup\binom{[n]}{a_s}$ by ``swapping" the membership of $s_i$ and $s_i+1$ for every set in $\cG$. More formally, for every $G_j\in\cG$, include $G_j'$ in $\cG'$, where
\[G_j'=\begin{cases}G_j&\text{ if }|G_j\cap\{s_i,s_i+1\}|\text{ is even}\\
    (G_j\setminus\{s_i\})\cup\{s_i+1\}&\text{ if }G_j\cap\{s_i,s_i+1\}=\{s_i\}\\
    (G_j\setminus\{s_i+1\})\cup\{s_i\}&\text{ if }G_j\cap\{s_i,s_i+1\}=\{s_i+1\}\\
    \end{cases}.\]
Then, $\cG'\subseteq\binom{[n]}{a_1}\cup\binom{[n]}{a_2}\cup\cdots\cup\binom{[n]}{a_s}$ as $|G_j|=|G_j'|$ for every $j\in[2^k]$. Moreover, since for every $j\in[2^k]$, $s_i\in G_j$ if and only if $s_i+1\in G_j'$, it is easy to verify $S'\in\osh(\cG')$ using Proposition~\ref{def:osh-direct}. 
\end{proof}

\begin{corollary}\label{cor:SL-order}
Let $0\le a_1<a_2<\cdots <a_s\le n$. If $S\in\osh\left(\binom{[n]}{a_1}\cup\binom{[n]}{a_2}\cup\cdots\cup\binom{[n]}{a_s}\right)$ and $S\preceq S'$, then $S'\in\osh\left(\binom{[n]}{a_1}\cup\binom{[n]}{a_2}\cup\cdots\cup\binom{[n]}{a_s}\right)$.
\end{corollary}
\begin{proof}
By Definition~\ref{def:prec}, since $S\preceq S'$, we may assume that the elements in $S$ and $S'$ are $s_1<s_2<\cdots<s_k$ and $s_1'<s_2'<\cdots<s_k'$, respectively, with $s_i\leq s_i'$ for every $i\in[k]$. Then, the result follows after repeatedly applying Lemma~\ref{lem:SL} to $i=k,k-1,\ldots,1$ in this order to increase $s_i$ to $s_i'$. 
\end{proof}

From now on, we consider $\osh\left(\binom{[n]}{a}\cup\binom{[n]}{a+d}\right)$ for some $0\le a\le n-d$ with $d\geq2$, as the $d=1$ case is a special case of Theorem~\ref{thm:ellwide}.
According to Corollary~\ref{cor:SL-order}, it suffices to determine all the \slm\ sets in $\osh\left(\binom{[n]}{a}\cup\binom{[n]}{a+d}\right)$.

The first step is to determine an upper bound on $\tm(S)$ if $S\in\osh\left(\binom{[n]}{a}\cup\binom{[n]}{a+d}\right)$.
\begin{proposition}\label{prop:tmin}
    Let $S\in\osh\left(\binom{[n]}{a}\cup\binom{[n]}{a+d}\right)$. If $\tm(S)\geq1$, then 
    \[\tm(S)\le M:=\min\{d,a+1,n-a-d+1,n-2d+2\}.\] 
\end{proposition}
\begin{proof}
Let the elements in $S$ be $s_1<s_2<\cdots<s_k$, and suppose $\tm(S)=t\ge 1$. Since $\osh\left(\binom{[n]}{a}\cup\binom{[n]}{a+d}\right)$ is a downset with respect to $\subseteq$, by removing elements if necessary, we may assume without loss of generality that $s_k=2k-t$.

Let $\cG=\{G_1,\ldots G_{2^{k}}\}$ be a family in standard order that order shatters $S$, with $G_1\cap S=\emptyset$ and $G_{2^{k}}\cap S=S$. From Definition~\ref{def:osh-orig}, $G_i\cap[s_k+1,n]$ does not depend on $i$, so $|G_i\cap[s_k]|$ can take on at most two distinct values with difference $d$. 

Let $r=|G_1\cap[s_k]|$. Then, $r\leq s_k-k=k-t$. Since $r\leq k-t<k\leq|G_{2^{k}}\cap[s_k]|$, we must have $|G_{2^{k}}\cap[s_k]|=r+d\geq k$. It follows that $t+r\leq k\leq d+r$, and so $t\leq d$. 

Moreover, $|G_1\cap[s_k]|=r$ and $|G_{2^{k}}\cap[s_k]|=r+d$ imply that there exist $i<j$ such that $|G_i\cap S|=|G_j\cap S|-1$, $|G_i\cap[s_k]|=r$, and $|G_j\cap[s_k]|=r+d$. It follows that $G_j\cap[s_k]$ has at least $d-1$ elements that are not in $S$, so $k-t=s_k-k=|[s_k]\setminus S|\geq d-1$, and thus $t\leq k-d+1$. Since $S$ is order shattered by sets of size $a$ and $a+d$, $|S|=k\leq a+d$, so $t\leq a+1$. 

Next, from $|G_1\cap[s_k]|=r$ and $|G_{2^{k}}\cap[s_k]|=r+d$, it also follows that $|G_1|=a$ and $|G_{2^{k}}|=a+d$. Since $G_1\cap S=\emptyset$, we must have $n-k\geq a$, so $k\leq n-a$. Similarly, using the definition of $G_i$ and $G_j$ above, it follows that $|G_i|=a$ and $|G_j|=a+d$, so $n-k\geq d-1$, and $k\leq n-d+1$. Combining these with $t\leq k-d+1$ from above, we get $t\leq n-a-d+1$ and $t\leq n-2d+2$, respectively. 
\end{proof}

It turns out that the behaviours of \slm\ sets in $\osh\left(\binom{[n]}{a}\cup\binom{[n]}{a+d}\right)$ differ depending on whether $d=2$ or $d>2$, so the rest of the proof is split into two separate sections.
\subsection{$d=2$}
When $d=2$, the \slm\ sets are characterized as follows.
\begin{theorem}\label{thm:d=2-min}
    Let $a\ge 0$ and $n\geq a+2$ be integers.
    The \slm\ sets in $\osh\left(\binom{[n]}{a}\cup\binom{[n]}{a+2}\right)$ are exactly the following.
    \begin{itemize}
        \item $\varnothing$ and $\{2\}$ with $\tm=0$.
        \item $\{2,3\}$ with $\tm=1$ if $n\geq a+3$.
        \item $S_k=\{2,3,4,6,\ldots ,2k\}$ with $\tm(S_k)=2$ for every $2\leq k\leq\min\{a+1,n-(a+1)\}$.
    \end{itemize}
\end{theorem}
\begin{proof}
    Without loss of generality we may assume that $a+1\le n-(a+1)$, as the other case follows by taking complements. For brevity, let $\cF=\binom{[n]}{a}\cup\binom{[n]}{a+2}$. By Proposition~\ref{prop:tmin}, if $S\in\osh(\cF)$ then $\tm(S)\le 2$. 

    First, we prove that the given sets are indeed order shattered by $\cF$. The case of $\varnothing$ is trivial. Let $A=[n-a+1,n]$. Then, $\{A,A\cup\{1,2\}\}\subseteq\binom{[n]}{a}\cup\binom{[n]}{a+2}=\cF$ order shatters $\{2\}$, and $\{A,A\cup\{1,2\},A\cup\{1,3\},A\cup\{2,3\}\}\subseteq\cF$ order shatters $\{2,3\}$ if $n\geq a+3$. Now consider $S_k$ with $2\leq k\leq a+1$. Let \[\cG=\{G:|G\cap[4]|\in\{1,3\}\text{ and }|G\cap\{2i-1,2i\}|=1\text{ for every }3\leq i\leq k\}.\] Note that the size of every $G\in\cG$ is either $k-1$ or $k+1$. Let $\widetilde{\cG}=\{G\cup[n-(a-k),n]:G\in\cG\}$, then it is easy to see that $\widetilde{\cG}$ order shatters $S_k$ and consists of sets of sizes $a$ and $a+2$.
    
    Now, suppose that $S=\{s_1,s_2,\ldots,s_k\}$ is a \slm\ set in $\osh(\cF)$. Then, $S\in\osh(\cF)$ implies that $k=|S|\leq a+2$. If $k=0$, then $S=\varnothing$, so suppose $k\geq1$. If $s_1=1$, then $\{1\}\in\osh(\cF)$. However, this is not possible by Definition~\ref{def:osh-orig} as $\cF=\binom{[n]}{a}\cup\binom{[n]}{a+2}$ does not contain two sets whose size difference is 1. Thus, $s_1\geq2$. If $|S|=1$, then $\{2\}\preceq S$, so $S=\{2\}$ by minimality. If $|S|\ge 2$ and $\tm(S)\leq1$, then $s_i\ge 2i-1$ for every $2\leq i\leq k$, so $S'=\{2,3,5,\ldots ,2k-1\}\preceq S$. If $k\geq 3$, then $S_{k-1}\prec S'$ is in $\osh(\cF)$ from above, so $S$ is not \slm, a contradiction. Hence, $k=2$ and $S=\{2,3\}$. If $|S|\geq 2$ and $\tm(S)=2$, then $k\ge 3$ and $s_i\ge 2i-2$ for every $i\in[k]$. Along with $s_1\geq2$ and $s_2\geq 3$, we have $S_{k-1}\preceq S$, so $S=S_{k-1}$ as otherwise it is not \slm.
    \end{proof}

\subsection{$d>2$}
Unfortunately, the case when $d>2$ is much more complicated. The \slm\ sets in $\osh\left(\binom{[n]}{a}\cup\binom{[n]}{a+d}\right)$ are characterized by the following two results.
\begin{theorem}\label{thm:d>2nonballot}
Let $d>2$ be an integer.
Then, the non-ballot \slm\ sets in $\osh\left(\binom{[n]}a\cup\binom{[n]}{a+d}\right)$ are exactly the following.
\begin{itemize}
    \item $S_{m,j}:=\{2,4,\ldots,2m,2m+d,2m+d+1,\ldots,2m+2d+j-1\}$, where $m,j\geq0$, and $\tm(S_{m,j})=j+1\leq M_1:=\min\{d,a-m+1,n-m-a-d+1,n-2m-2d+2\}$.
    \item $S'_{m,r}:=\{2,4,\ldots,2m,2m+d,2m+d+1,\ldots,2m+3d-2,2m+3d,\ldots,2m+3d+2r\}$, where $m,r\geq 0$, and $\tm(S_{m,r}')=d\leq M_2:=\min\{a-m-r,n-m-r-a-d,n-2m-2r-2d\}$.
\end{itemize}
\end{theorem}

\begin{theorem}\label{thm:d>2ballot}
Let $d>2$ be an integer. Let $M:=\max\{\min\{a,n-a\},\min\{a+d,n-a-d\}\}$. 
Then, the ballot \slm\ sets in $\osh\left(\binom{[n]}a\cup\binom{[n]}{a+d}\right)$ are exactly the following.
\begin{itemize}
    \item $B_m:=\{2,4,\ldots,2m\}$ for every $0\leq m\leq M$.
    \item In addition, if $a<n/2<a+d$, $B_{m,j}:=\{2,4,\ldots,2m,2m+d,2m+d+1,\ldots,2m+d+j-1\}$ for all $0\leq m\leq a$ and $j\geq 1$ satisfying $m+j>M=\max\{a,n-a-d\}$ and $2m+2j\leq 2m+d+j-1\leq n$.
\end{itemize}
\end{theorem}

We first show that the sets mentioned in Theorem~\ref{thm:d>2nonballot} are indeed order shattered by $\binom{[n]}{a}\cup\binom{[n]}{a+d}$.
\begin{proposition}\label{prop:oshevenconstruction}
Let $d>2$ be an integer. 
Let $\cF=\binom{[n]}a\cup\binom{[n]}{a+d}$.
\begin{itemize}
    \item Let $S_{m,j}=\{2,4,\ldots,2m,2m+d,2m+d+1,\ldots,2m+2d+j-1\}$, where $m,j\geq0$. Then, $S_{m,j}\in\osh(\cF)$ if  $\tm(S_{m,j})=j+1\leq M_1:=\min\{d,a-m+1,n-m-a-d+1,n-2m-2d+2\}$. 
    \item Let $S'_{m,r}=\{2,4,\ldots,2m,2m+d,2m+d+1,\ldots,2m+3d-2,2m+3d,\ldots,2m+3d+2r\}$, where $m,r\geq 0$. Then, $S_{m,r}'\in\osh(\cF)$ if $d\leq M_2:=\min\{a-m-r,n-m-r-a-d,n-2m-2r-2d\}$. 
\end{itemize}
\end{proposition}
\begin{proof}
We begin with the following claim. 
\begin{claim}\label{claim:even}
For every $m\geq 0$, there exists $\cG'\subseteq\binom{[2m]}{m}$ with $|\cG'|=2^m$ such that $\{2,4,\ldots,2m\}$ is order shattered by $\cG'$.
\end{claim}
\begin{proof}[Proof of Claim~\ref{claim:even}]
Let $\cG'=\{G\subseteq[2m]:|G\cap\{2i-1,2i\}|=1\text{ for each }i\in[m]\}\subseteq\binom{[2m]}{m}$. Then, it is easy to verify from Definition~\ref{def:osh-orig} that $\{2,4,\ldots,2m\}\in\osh(\cG')$. 
\renewcommand{\qedsymbol}{$\boxdot$}
\end{proof}
\renewcommand{\qedsymbol}{$\square$}

First suppose $j+1\leq M_1$, we want to show that $S_{m,j}\in\osh(\cF)$. Since $j\leq M_1-1\leq\min\{n-m-a-d,n-2m-2d+1\}$, we have $(d-1)+(n-2m-2d-j+1)=n-2m-d-j\geq a-m$, and $n-2m-2d-j+1\geq 0$. Thus, we can find $0\leq x\leq d-1$ and $0\leq y\leq n-2m-2d-j+1$ such that $x+y=a-m$. Moreover, we can ensure that if $y>0$, then $x=d-1$. We claim that this implies $x\geq j$. Indeed, if $y=0$, then $x=a-m\geq M_1-1\geq j$ from assumption. If $y>0$ instead, then $x=d-1\geq M_1-1\geq j$, as required.

Define $\cG\subseteq\cF$ as follows. Fix $J\subseteq[2m+2d+j,n]$ with $|J|=y$, which is possible as $y\leq n-2m-2d-j+1$. For each $I\subseteq[2m+d,2m+2d+j-1]$ with $|I|\leq x$, pick $G_I\subseteq[2m+1,2m+d-1]$ with $|G_I|=x-|I|$, which is possible as $x\leq d-1$. For each $I\subseteq[2m+d,2m+2d+j-1]$ with $x+1\leq|I|\leq d+j$, pick $G_I\subseteq[2m+1,2m+d-1]$ with $|G_I|=x+d-|I|$, which is possible as $0\leq x-j\leq x+d-|I|\leq d-1$. Then, let $\cG'$ be given by Claim~\ref{claim:even}, and let
\[\cG=\bigcup_{I\subseteq[2m+d,2m+2d+j-1]}\{G'\cup G_I\cup I\cup J:G'\in\cG'\}.\]
Note that $|\cG|=2^{m+d+j}$, $\cG\subseteq\cF$, and $S_{m,j}\in\osh(\cG)$. 

Now suppose $d\leq M_2$, we want to show $S_{m,r}'\in\osh(\cF)$. Similar to above, using $d\leq M_2\leq\min\{n-m-r-a-d,n-2m-2r-2d\}$, we can find $0\leq x\leq d-1$ and $0\leq y\leq n-2m-3d-2r$, such that $x+y=a-m-r-1$, and $x=d-1$ if $y>0$. Also, if $y=0$, then $x=a-m-r-1\geq M_2-1\geq d-1$.

Fix $J\subseteq[2m+3d+2r+1,n]$ with $|J|=y$. For each $I\subseteq[2m+d,2m+3d-2]$ with $|I|\leq x$, pick $G_I\subseteq[2m+1,2m+d-1]$ with $|G_I|=x-|I|$, which is possible as $x\leq d-1$. For each $I\subseteq[2m+d,2m+3d-2]$ with $x+1\leq|I|\leq 2d-1$, pick $G_I\subseteq[2m+1,2m+d-1]$ with $|G_I|=x+d-|I|$, which is possible as $0\leq x-d+1\leq x+d-|I|\leq d-1$. Let $\cG_1'=\{G\subseteq[2m]:|G\cap\{2i-1,2i\}|=1\text{ for each }i\in[m]\}$ be the family given by Claim~\ref{claim:even}, and let $\cG_2'=\{G\subseteq[2m+3d+2r]\setminus[2m+3d-2]:|G\cap\{2m+3d+2i-1,2m+3d+2i\}|=1\text{ for each }0\leq i\leq r\}$. Then, let
\[\cG=\bigcup_{I\subseteq[2m+d,2m+3d-2]}\{G_1'\cup G_I\cup I\cup G_2'\cup J:G_1'\in\cG_1', G_2'\in\cG_2'\}.\]
Note that $|\cG|=2^{m+2d+r}$, $\cG\subseteq\cF$, and $S_{m,r}'\in\osh(\cG)$.  
\end{proof}

Now we show that the sets in Theorem~\ref{thm:d>2ballot} are order shattered by $\binom{[n]}{a}\cup\binom{[n]}{a+d}$. By the $\ell=1$ case of Theorem~\ref{thm:ellwide}, $\binom{[n]}a\cup\binom{[n]}{a+d}$ order shatters all ballot sets of size up to $M$, which includes $B_m$ for all $0\leq m\leq M$. 

When $a<n/2<a+d$, sets of the form $B_{m,j}$ are handled by the following proposition. 
\begin{proposition}\label{prop:ballotconstruction}
Let $d>2$ be an integer. Suppose $a<n/2<a+d$, and let $M=\max\{a,n-a-d\}$. If $0\leq m\leq a$ and $j\geq 1$ satisfy $m+j>M$ and $2m+2j\leq 2m+d+j-1\leq n$, then $B_{m,j}=\{2,4,\ldots,2m,2m+d,2m+d+1,\ldots,2m+d+j-1\}\in\osh\left(\binom{[n]}a\cup\binom{[n]}{a+d}\right)$.
\end{proposition}
\begin{proof}
Note that $2m+d+j-1\geq 2m+2j$ implies that $j\leq d-1$. Also, $m+j>M\geq a$ implies that $a-m\leq j-1<d-1$. 

Define $\cG\subseteq\binom{[n]}a\cup\binom{[n]}{a+d}$ as follows. For each $I\subseteq[2m+d,2m+d+j-1]$ with $|I|\leq a-m$, pick $G_I\subseteq[2m+1,2m+d-1]$ with $|G_I|=a-m-|I|$, which is possible as $0\leq a-m-|I|\leq a-m<d-1$. For each $I\subseteq[2m+d,2m+d+j-1]$ with $a-m+1\leq|I|\leq j$, pick $G_I\subseteq[2m+1,2m+d-1]$ with $|G_I|=a+d-m-|I|$, which is possible as \[1\leq a-m+1\leq a+d-m-j\leq a+d-m-|I|\leq a+d-m-(a-m+1)=d-1.\] Then, let $\cG'$ be given by Claim~\ref{claim:even}, and let
\[\cG=\bigcup_{I\subseteq[2m+d,2m+d+j-1]}\{G'\cup G_I\cup I:G'\in\cG'\}.\]
Note that $|\cG|=2^{m+j}$, $\cG\subseteq\binom{[n]}a\cup\binom{[n]}{a+d}$, and $B_{m,j}\in\osh(\cG)$. 
\end{proof}

Next, we prove the following useful result which says that if $S$ is order shattered by a family $\cG=\{G_1,G_2,\ldots ,G_{2^k}\}\subseteq\binom{[n]}{a}\cup\binom{[n]}{a+d}$ in standard order, and $S$ begins with consecutive even numbers $2,4,\ldots,2m$, then $|G_i\cap[2m]|=m$ for every $i\in[2^k]$.  

\begin{lemma}\label{lemma:equal-sizes}
Let $m,j\ge 0$, $d>2$, and $1\le x_1<x_2<\cdots <x_j$ be integers. Suppose $S=\{2,4,\ldots ,2m, 2m+x_1,2m+x_2,\ldots ,2m+x_j\}$ is order shattered by a family $\{G_1,G_2,\ldots ,G_{2^{m+j}}\}\subseteq\binom{[n]}{a}\cup\binom{[n]}{a+d}$ in standard order. Then, $|G_i\cap [2m]|=m$ for every $i\in[2^{m+j}]$.
\end{lemma}
\begin{proof}
We use induction on $k$ to prove that for every $0\leq k\leq m$, $|G_i\cap [2k]|=k$ for every $i\in[2^{m+j}]$. The base case when $k=0$ is trivial. Suppose now that $k\geq1$, then by induction hypothesis, $|G_i\cap [2k-2]|=k-1$ for every $i\in[2^{m+j}]$. 
    Observe that by Proposition~\ref{def:osh-direct},
    \[2k\in G_i\text{ if and only if }i\in\bigcup_{\ell=1}^{2^{m-k+j}}[\ell\cdot 2^k-2^{k-1}+1,\ell\cdot 2^k],\] and furthermore for every $\ell\in[2^{m-k+j}]$ and all $(\ell-1)2^k<i<i'\le\ell\cdot 2^k$,
    \[G_i\cap[2k+1,n]=G_{i'}\cap[2k+1,n].\]
    It follows that for every $\ell\in[2^{m-k+j}]$ and all $(\ell-1)2^k<i\le \ell\cdot2^k-2^{k-1}<i'\le \ell\cdot 2^k$, \[|G_i\cap([n]\setminus\{2k-1\})|=|G_{i'}\cap([n]\setminus\{2k-1\})|-1,\] 
    and so $||G_i|-|G_{i'}||\leq2$. However, since $|G_i|,|G_{i'}|\in\{a,a+d\}$ and $d>2$, this forces $|G_i|=|G_{i'}|$, so $G_i$ contains $2k-1$, while $G_{i'}$ does not. Therefore, $|G_i\cap[2k]|=k$ for every $i\in[2^{m+j}]$, as required.
\end{proof}

The next two results show that if $S$ is a \slm\ set order shattered by $\binom{[n]}{a}\cup\binom{[n]}{a+d}$, and $m\geq0$ is the maximum number such that $S$ begins with consecutive even numbers $2,4,\ldots,2m$, then $m\leq a$ and the next number in $S$, if it exists, is at least $2m+d$. 

\begin{lemma}\label{lemma:1gap}
Let $m\ge 0$ and $d>2$ be integers. Then $S=\{2,4,\ldots ,2m, 2m+1\}$ is not order shattered by $\binom{[n]}{a}\cup\binom{[n]}{a+d}$.
\end{lemma}
\begin{proof}
Suppose for a contradiction that $\cG=\{G_1,\ldots,G_{2^{m+1}}\}\subseteq\binom{[n]}{a}\cup\binom{[n]}{a+d}$ is a family that order shatters $S$. By Lemma~\ref{lemma:equal-sizes}, $|G_i\cap[2m]|=m$ for every $i\in[2^{m+1}]$. Furthermore, there exists $1\leq i,j\leq 2^{m+1}$ such that $2m+1\in G_j\setminus G_i$. However, this implies that $|G_j|=|G_i|+1$, which is impossible as $d>2$. 
\end{proof}

\begin{lemma}\label{lemma:dgap}
Let $m\ge 0$, $j\geq1$, $d>2$, and $3\le x_1<x_2<\cdots <x_j$ be integers. If $S=\{2,4,\ldots ,2m, 2m+x_1,2m+x_2,\ldots ,2m+x_j\}$ is a \slm\ set order shattered by $\binom{[n]}{a}\cup\binom{[n]}{a+d}$, then $m\leq a$ and $x_1\geq d$.
\end{lemma}
\begin{proof}
    Let $\cG=\{G_1,\ldots,G_{2^{m+j}}\}\subseteq\binom{[n]}a\cup\binom{[n]}{a+d}$ be a family in standard order that order shatters $S$. By Lemma~\ref{lemma:equal-sizes}, $|G_i\cap[2m]|=m$ for every $i\in[2^{m+j}]$. If $m>a$, then for every $i\in[2^{m+j}]$, $|G_i|>a$, so $|G_i|=a+d$. This implies that  $S$ is a ballot set of size at most $M$ using the $\ell=1$ case of Theorem~\ref{thm:ellwide}, and so $B_{m+j}=\{2,4,\ldots ,2m+2j\}\prec S$. Furthermore, Theorem~\ref{thm:ellwide} also implies that $B_{m+j}$ is order shattered by $\binom{[n]}{a}\cup\binom{[n]}{a+d}$, contradicting that $S$ is \slm. Thus, $m\leq a$. Also, by Definition~\ref{def:osh-orig}, $|G_i\cap([2m+x_j+1,n])|$ does not depend on the choice of $i\in[2^{m+j}]$. 

Suppose for a contradiction that $x_1<d$. For simplicity, let $x_0=0$. Let $2\leq j'\leq j$ be the minimal index satisfying $2m+x_{j'}\geq 2m+2j'+d-2$ if such an index exists, and let $j'=j+1$ otherwise. 

We use induction on $k$ to show that for every $1\leq k<j'$ and every $i\in[2^{m+j}]$, $m+k\leq|G_i\cap[2m+x_k]|\leq m+x_k-k$. In particular, $x_k\geq 2k$ for every $1\leq k<j'$. Indeed, for every $i\in[2^{m+j}]$, either by induction hypothesis when $k>1$, or using $|G_i\cap[2m]|=m$ from above when $k=1$, we have $m+k-1\leq |G_i\cap[2m+x_{k-1}]|\leq m+x_{k-1}-k+1$. Let $i\in[2^{m+j}]$. Suppose first that $2m+x_k\in G_i$. By Proposition~\ref{def:osh-direct}, if we set $i'=i-2^{m+k-1}$, then $G_i\cap S$ and $G_{i'}\cap S$ only differ at $2m+x_k$, and $G_i\cap[2m+x_k+1,n]=G_{i'}\cap[2m+x_k+1,n]$. Hence, 
\begin{align*}
&\phantom{{}\leq{}}|G_i|-|G_{i'}|=|G_i\cap[2m+x_k]|-|G_{i'}\cap[2m+x_k]|\\
&\leq|G_i\cap[2m+x_{k-1}]|-|G_{i'}\cap[2m+x_{k-1}]|+|G_i\cap[2m+x_{k-1}+1,2m+x_k]|\\
&\leq(m+x_{k-1}-k+1)-(m+k-1)+(x_k-x_{k-1})\\
&=x_k-2k+2<d,
\end{align*}
where the second inequality follows from the induction hypothesis, and the last inequality follows from $k<j'$ and the minimality of $j'$. Similarly, $|G_{i'}|-|G_i|<d$, so $|G_i|=|G_{i'}|$ as they can only have size $a$ or $a+d$. In particular, $|G_i\cap[2m+x_k]|=|G_{i'}\cap[2m+x_k]|$. Therefore, $|G_i\cap[2m+x_k]|\geq1+|G_i\cap[2m+x_{k-1}]|\geq m+k$, and $|G_i\cap[2m+x_k]|=|G_{i'}\cap[2m+x_k]|=|G_{i'}\cap[2m+x_{k-1}]|+|G_{i'}\cap[2m+x_{k-1}+1,2m+x_k]|\leq m+x_{k-1}-(k-1)+(x_k-x_{k-1}-1)=m+x_k-k$, as required. The case when $2m+x_k\not\in G_i$ is the same after letting $i'=i+2^{m+k-1}$ and exchanging the roles of $i'$ and $i$ in the argument above. This completes the induction.

Now, if $j'=j+1$, then applying the result above for $k=j-1$ implies that $m+j-1\leq|G_i\cap[2m+x_{j-1}]|\leq m+x_{j-1}-j+1$. For any distinct $i,i'\in[2^{m+j}]$, $|G_i\cap[2m+x_j+1,n]|=|G_{i'}\cap[2m+x_j+1,n]|$ by Definition~\ref{def:osh-orig}, so by a similar calculation, $||G_i|-|G_{i'}||\leq(m+x_{j-1}-j+1)-(m+j-1)+(x_j-x_{j-1})=x_j-2j+2<d$. Hence, $|G_i|=|G_{i'}|$ for any distinct $i,i'\in[2^{m+j}]$, so $S$ is a ballot set of size at most $M$ by the $\ell=1$ case of Theorem~\ref{thm:ellwide}. Like above, this implies that $B_{m+j}\prec S$, contradicting that $S$ is \slm.

Suppose instead that $j'\leq j$. Then, by applying the result above for $k=j'-1$, we get $|G_i|\geq|G_i\cap[2m+x_{j'-1}]|\geq m+j'-1$ for every $i\in[2^{m+k}]$. Again, if $m+j'-1>a$, then $S$ is a
ballot set of size at most $M$ by the $\ell=1$ case of Theorem~\ref{thm:ellwide}, leading to a contradiction. Hence, $m+j'-1\leq a$. We now split into three cases depending on the size of $j-j'$.

\textbf{Case 1.} $j-j'\leq d-2$. Then, using $j'\geq2$, the definition of $j'$, and $x_k\geq 2k$ for every $1\leq k<j'$, we have $B_{m+j'-1,j-j'+1}=\{2,4,\ldots,2m+2j'-2,2m+2j'+d-2,2m+2j'+d-1,\ldots,2m+j'+d+j-2\}\prec S$. Since $j-j'\leq d-2$, we have $2m+j'+d+j-2\geq 2m+2j$, so $S$ is a ballot set. If $m+j\leq M$, then $B_{m+j}\prec S$, contradicting the minimality of $S$, so $m+j>M$. If $a+d\leq n/2$, then no set of size larger than $a+d=M$ can be order shattered by $\binom{[n]}a\cup\binom{[n]}{a+d}$. If $a\geq n/2$, then no set of size larger than $n-a=M$ can be order shattered by $\binom{[n]}a\cup\binom{[n]}{a+d}$. Therefore, $a<n/2<a+d$, and $B_{m+j'-1,j-j'+1}\in\osh\left(\binom{[n]}{a}\cup\binom{[n]}{a+d}\right)$ by Proposition~\ref{prop:ballotconstruction}, which contradicts that $S$ is \slm.

\textbf{Case 2.} $d-1\le j-j'\le 2d-2$. Consider the set $S_{m+j'-1,j-j'-d+1}$ in Proposition~\ref{prop:oshevenconstruction}, noting that $S_{m+j'-1,j-j'-d+1}\prec S$. First, $j-j'-d+1\leq d-1$. Next, since $S$ is order shattered by $\binom{[n]}{a}\cup\binom{[n]}{a+d}$, we have $m+j=|S|\leq\min\{a+d,n-a\}$, from which it follows that $j-j'-d+1\le a-(m+j'-1)$ and $j-j'-d+1\le n-(m+j'-1)-a-d$. Finally, from $2m+j+j'+d-2\leq s_{m+j}\le n$, we get $j-j'-d+1\leq n-2(m+j'-1)-2d+1$. Putting all of these together gives $j-j'-d+1\leq M_1$, so $S_{m+j'-1,j-j'-d+1}$ is order shattered by $\binom{[n]}{a}\cup\binom{[n]}{a+d}$ by Proposition~\ref{prop:oshevenconstruction}, and $S_{m+j'-1,j-j'-d+1}\prec S$, contradicting the minimality of $S$.

\textbf{Case 3.} $j-j'\geq 2d-1$. By Proposition~\ref{prop:tmin}, $\tm(S)\leq d$, so $s_{m+j'+i}\geq2(m+j'+i)-d$ for every $2d-1\leq i\leq j-j'$. Let $r=j-j'-2d+1$ and consider the set $S'_{m+j'-1,r}$ in Proposition~\ref{prop:oshevenconstruction}, noting that $S'_{m+j'-1,r}\prec S$. Similarly, using $m+j=|S|\leq\min\{a+d,n-a\}$, we get $a-(m+j'-1)-r\geq d$ and $n-(m+j'-1)-r-a-d\geq d$. Furthermore, using $2m+2j-d\leq s_{m+j}\le n$, we have $n-2(m+j'-1)-2r-2d\geq d$. Therefore, $d\leq M_2$, and so $S'_{m+j'-1,r}$ is order shattered by $\binom{[n]}{a}\cup\binom{[n]}{a+d}$ by Proposition~\ref{prop:oshevenconstruction}, and $S'_{m+j'-1,r}\prec S$, contradicting the minimality of $S$.
\end{proof}

Finally, we can put everything together to prove Theorem~\ref{thm:d>2nonballot} and Theorem~\ref{thm:d>2ballot}.
\begin{proof}[Proof of Theorem~\ref{thm:d>2nonballot}]
For brevity, denote $\binom{[n]}a\cup\binom{[n]}{a+d}$ by $\cF$. By Proposition~\ref{prop:oshevenconstruction}, each $S_{m,j}$ and $S'_{m,r}$ satisfying the condition in Theorem~\ref{thm:d>2nonballot} is in $\osh(\cF)$.

First, we show that each $S_{m,j}$ is a \slm\ set in $\osh(\cF)$. Suppose for a contradiction that there exists a \slm\ set $T$ in $\osh(\cF)$ such that $T\prec S_{m,j}$. From the definition of $\preceq$, $|T|=|S_{m,j}|=m+d+j$. Let the elements in $T$ be $t_1<t_2<\cdots<t_{m+d+j}$, and let the elements in $S_{m,j}$ be $s_1<s_2<\cdots<s_{m+d+j}$. Since $T\prec S_{m,j}$, there exists a minimal $i\in[m+d+j]$ such that $t_i<s_i$. If $i\in[m]$, then $\{2,4,\ldots,2i-2,2i-1\}\in\osh(\cF)$ as it is a subset of $T$. However, this contradicts Lemma~\ref{lemma:1gap}. If $i>m$ instead, then observe that $i=m+1$ because $s_{m+1},s_{m+2},\ldots,s_{m+d+j}$ are consecutive integers. Thus, $t_i=2i$ for all $i\in[m]$, and $t_{m+1}<2m+d$. 
Since $S_{m,j}$ is not a ballot set, $T$ cannot be a ballot set. In particular, we can find a minimum index $m+1\leq k\leq m+d+j$ such that $t_k\not=2k$. If $t_k=2k-1$, then this contradicts Lemma~\ref{lemma:1gap}, so $t_k\geq 2k+1\geq t_{k-1}+3$. Then, Lemma~\ref{lemma:dgap} implies that $t_k\geq t_{k-1}+d=2k+d-2$. If $k=m+1$, then $t_{m+1}\geq2m+d$, contradiction. If $k\geq m+2$, then $t_k\geq 2k+d-2>m+k+d-1=s_k$, contradicting $T\prec S_{m,j}$.

Next, we show that each $S'_{m,r}$ is a \slm\ set in $\osh(\cF)$. Suppose for a contradiction that there exists a \slm\ set $T$ in $\osh(\cF)$ such that $T\prec S'_{m,r}$. Then, $|T|=|S'_{m,r}|=m+2d+r$. Let the elements in $T$ be $t_1<t_2<\cdots<t_{m+2d+r}$, and let the elements in $S'_{m,r}$ be $s_1<s_2<\cdots<s_{m+2d+r}$. Since $T\prec S'_{m,r}$, there exists a minimal $i\in[m+2d+r]$ such that $t_i<s_i$. Like above, if $i\in[m]$, then $\{2,4,\ldots,2i-2,2i-1\}\in\osh(\cF)$, which contradicts Lemma~\ref{lemma:1gap}. If instead $m<i\leq m+2d-1$, then $i=m+1$ and $t_{m+1}<2m+d$. Like above, let $m+1\leq k\leq m+2d+r$ be the minimum index such that $t_k\not=2k$, which exists as $T\prec S_{m,r}'$ is not a ballot set. Note that $t_k=2k-1$ contradicts Lemma~\ref{lemma:1gap}, so $t_k\geq 2k+d-2$ by Lemma~\ref{lemma:dgap}. If $k=m+1$, then $t_{m+1}\geq 2m+d$, a contradiction. If $k\geq m+2$, then as $s_i<2i$ for every $m+d\leq i\leq m+2d+r$, we must have $k\leq m+d$. But then, $t_k\geq 2k+d-2>k+m+d-1=s_k$, a contradiction. Finally, if $i\geq m+2d$, then $t_i\leq s_i-1=2i-d-1$, so $T$ is not a $d$-ballot set, contradicting Proposition~\ref{prop:tmin}.

Now for the reverse, let $S$ be a non-ballot \slm\ set in $\osh(\cF)$. Let the elements in $S$ be $s_1<\cdots<s_\ell$ with $\ell=|S|$, and let $m\in[\ell]$ be maximal such that $s_i=2i$ for every $i\in[m]$. We claim that $s_{m+1}\geq 2m+d$. Indeed, the maximality of $m$ implies that $s_{m+1}\not=2m+2$. If $s_{m+1}=2m+1$, then $\{2,4,\ldots,2m,2m+1\}\in\osh(\cF)$, which contradicts Lemma~\ref{lemma:1gap}. If $2m+3\leq s_{m+1}<2m+d$, then this contradicts Lemma~\ref{lemma:dgap}.

Next, let $j=\tm(S)-1\geq0$. Then, there exists $m+1<h\leq\ell$ such that $s_h=2h-j-1$. But $s_h\geq s_{m+1}+(h-m-1)\geq m+d+h-1$, so $\ell\geq h\geq m+d+j$. 

If $\ell\leq m+2d-1$, then $S\succeq\{2,4,\ldots,2m,2m+d,\ldots,2m+d+\ell-m-1\}=S_{m,\ell-m-d}$. Since $S\in\osh\left(\binom{n}{a}\cup\binom{n}{a+d}\right)$, $\ell=|S|\leq\min\{a+d,n-a\}$. Combined with $\ell\geq m+d+j$ and $2m+d+\ell-m-1\leq n$, we get $\ell-m-d+1\leq M_1$. Hence, $S_{m,\ell-m-d}\in\osh(\cF)$ by Proposition~\ref{prop:oshevenconstruction}. But $S$ is $\preceq$-minimal, so $S=S_{m,\ell-m-d}$, with $\ell-m-d=j$ since $\tm(S)=j+1$.

If $\ell\geq m+2d$, set $r=\ell-m-2d$. Then, using $\tm(S)\leq d$, we get $S\succeq\{2,4,\ldots,2m,2m+d,2m+d+1,\ldots,2m+3d-2,2m+3d,\ldots,2m+3d+2r\}=S'_{m,r}$. Using $\ell=|S|\leq\min\{a+d,n-a\}$ and $2m+3d+2r\leq n$, we get $d\leq M_2$, so $S'_{m,r}\in\osh(\cF)$ by Proposition~\ref{prop:oshevenconstruction}. Therefore, the minimality of $S$ implies that $S=S'_{m,r}$.
\end{proof}

\begin{proof}[Proof of Theorem~\ref{thm:d>2ballot}]
Recall that by the $\ell=1$ case of Theorem~\ref{thm:ellwide}, $\binom{[n]}a\cup\binom{[n]}{a+d}$ order shatters all ballot sets of size up to $M$. Combined with Lemma~\ref{lemma:1gap}, the ballot \slm\ sets of size at most $M$ in $\osh\left(\binom{[n]}a\cup\binom{[n]}{a+d}\right)$ are exactly $\{B_m:0\leq m\leq M\}$. Thus, to prove Theorem~\ref{thm:d>2ballot}, it remains to determine the ballot \slm\ sets in $\osh\left(\binom{[n]}a\cup\binom{[n]}{a+d}\right)$ of size larger than $M$. If $a+d\leq n/2$, then no set in $\osh\left(\binom{[n]}a\cup\binom{[n]}{a+d}\right)$ has size larger than $a+d=M$. Similarly, if $a\geq n/2$, then no set in $\osh\left(\binom{[n]}a\cup\binom{[n]}{a+d}\right)$ has size larger than $n-a=M$. Therefore, it suffices to consider the case when $a<n/2<a+d$, and $M=\max\{a,n-a-d\}$.

In this case, by Proposition~\ref{prop:ballotconstruction}, if $0\leq m\leq a$ and $j\geq 1$ satisfy $m+j>M$ and $2m+2j\leq 2m+d+j-1\leq n$, then $B_{m,j}=\{2,4,\ldots,2m,2m+d,2m+d+1,\ldots,2m+d+j-1\}\in\osh\left(\binom{[n]}a\cup\binom{[n]}{a+d}\right)$. To finish the proof, we need to show they are \slm\ and they are the only ballot \slm\ sets of size larger than $M$. 

First, suppose there exists a \slm\ set $B'$ in $\osh\left(\binom{[n]}a\cup\binom{[n]}{a+d}\right)$ such that $B'\prec B_{m,j}$. 
Let the elements in $B'$ be $b_1'<\cdots<b_{m+j}'$ and let the elements in $B_{m,j}$ be $b_1<\cdots<b_{m+j}$. Let $i\in[m+j]$ be minimal such that $b_i'<b_i$. If $i\in[m]$, then $b_i'<2i$, contradicting Lemma~\ref{lemma:1gap}. If $i>m$, then $i=m+1$ as $b_{m+1},\ldots,b_{m+j}$ are consecutive integers. Since $|B'|=m+j>M$, we cannot have $b_k'=2k$ for every $k\in[m+j]$, so we can let $m+1\leq k\leq m+j$ be minimal such that $b_k'>2k$. Then, Lemma~\ref{lemma:dgap} implies that $b_k'\geq b_{k-1}'+d=2k-2+d$. If $k=m+1$, then $b_{m+1}'\geq2m+d=b_{m+1}$, a contradiction. If $k\geq m+2$, then $b_{m+j}'\geq b_k'+(m+j-k)\geq m+j+k-2+d\geq 2m+j+d$. However, $b_{m+j}'\leq b_{m+j}=2m+d+j-1$, a contradiction. Thus, such $B'$ does not exist, so $B_{m,j}$ is \slm.

Now, let $B$ be a ballot \slm\ set in $\osh\left(\binom{[n]}a\cup\binom{[n]}{a+d}\right)$ with $|B|>M$. Let the elements in $B$ be $b_1<\cdots<b_\ell$, let $0\leq m\leq\ell$ be maximum such that $b_i=2i$ for every $i\in[m]$, and let $j=\ell-m$. Then, $m\leq a$ by Lemma~\ref{lemma:equal-sizes} and the $\ell=1$ case of Theorem~\ref{thm:ellwide}. Thus, $j\geq 1$ as $m+j>M$. By Lemma~\ref{lemma:dgap}, $b_{m+1}\geq 2m+d$, and so $b_{m+i}\geq 2m+d+i-1$ for every $i\in[j]$. 

We now consider several cases depending on the size of $j$.

\textbf{Case 1.} $j\le d-1$. Then, $2m+2j\leq 2m+j+d-1\leq b_{m+j}\leq n$, so $B_{m,j}$ is order shattered by $\binom{[n]}a\cup\binom{[n]}{a+d}$ by Proposition~\ref{prop:ballotconstruction} and $B_{m,j}\preceq B$. Since $B$ is \slm, $B=B_{m,j}$.

\textbf{Case 2.} $j\geq d$. In this case we show that $B$ is not \slm. Since $B$ is a ballot set in $[n]$, $b_{m+i}\geq 2m+2i$ for every $d\leq i\leq j$, and thus $B^*=\{2,4,\ldots,2m,2m+d,2m+d+1,\ldots,2m+2d-2,2m+2d,2m+2d+2,\ldots,2m+2j\}\preceq B$. 

If $j\le 2d-1$, let $j'=j-d$ and consider the set $S_{m,j'}$ in Proposition~\ref{prop:oshevenconstruction}. Since $j\le 2d-1$, we have $j'\le d-1$. Furthermore, from $m+j=\ell=|B|\leq\min\{n-a,a+d\}$, we have $j'\le a-m$ and $j'\le n-m-a-d$. Finally, as $2m+2j\leq b_{m+j}\le n$ and $d\le j$, we conclude that $j\le n-2m-j\le n-2m-d$, and so $j'\le n-2m-2d$. Putting all of these together, we have $j'+1\leq M_1$, so $S_{m,j'}$ is order shattered by $\binom{[n]}a\cup\binom{[n]}{a+d}$ by Proposition~\ref{prop:oshevenconstruction}, and $S_{m,j'}\prec B^*\preceq B$, a contradiction.

If $j\geq 2d$, let $r=j-2d$ and consider the set $S'_{m,r}$ in Proposition~\ref{prop:oshevenconstruction}. Again, using $m+j\leq\min\{n-a,a+d\}$, we have $a-m-r=a-m-j+2d\geq d$ and $n-m-r-a-d=n-m-j-a+d\geq d$. Furthermore, using $2m+2j\leq n$, we have $n-2m-2r-2d=n-2m-2j+2d\geq d$. Therefore, $d\leq M_2$, and so $S'_{m,r}$ is order shattered by $\binom{[n]}a\cup\binom{[n]}{a+d}$ by Proposition~\ref{prop:oshevenconstruction}, and $S'_{m,r}\prec B^*\preceq B$, a contradiction.
\end{proof}

\bibliographystyle{alpha}
\bibliography{osh}

\end{document}